\documentclass{llncs}
\makeatletter

\let\proof\@undefined
\let\endproof\@undefined
\makeatother
\usepackage{graphicx}
\usepackage{amsmath}
\usepackage{amssymb}
\usepackage{amsthm}
\usepackage{float}
\usepackage{caption}
\usepackage{algorithmic}
\usepackage{algorithm}
\floatname{algorithm}{Procedure}

\newtheorem{prop}{Property}
\begin{document}
\title{Partitioning in the space of antimonotonic functions}
\author{Patrick De Causmaecker and Stefan De Wannemacker}
\institute{CODeS \& ITEC, KUL@K, Katholieke Universiteit Leuven, Kortrijk, Belgium
}
\maketitle
\begin{abstract}
This paper studies partitions in the space of antimonotonic boolean functions on sets of $n$ elements.
The antimonotonic functions are the antichains of the partially ordered set of subsets.
We analyse and characterise a natural partial ordering on this set.
We study the intervals according to this ordering.
We show how intervals of antimonotonic functions, and a fortiori the whole space of antimonotonic functions
can be partitioned as disjoint unions of certain classes of intervals.
These intervals are uniquely determined by antimonotonic functions on smaller sets.
This leads to recursive enumeration algorithms  and new recursion relations.
Using various decompositions, we derive new recursion formulae for the number of antimonotonic functions and hence for the number of monotonic functions (i.e. the Dedekind number).
\end{abstract}
\section{Introduction}
The $n$th Dedekind number counts the number of antichains of subsets of an $n$-element set or the number of elements in a free distributive lattice on $n$  generators. Equivalently, it counts the number of monotonic functions on the subsets of a finite set of $n$ elements \cite{DEDEKIND,SLOANE}.
In 1969, Kleitman \cite{KLEITMAN} obtained an upper bound on the logarithm of the $n$th Dedekind number which was later  improved by Kleitman and Markowsky \cite{KLEITMAN_MARKOWSKY} in 1975, namely
\[  (1+O((log n)/n)) \binom{n}{\lfloor n/2 \rfloor}.\]
In 1981, Korshunov \cite{KORSHUNOV} used a more complicated approach to give asymptotics for the $n$th Dedekind number itself. All these proofs were simplified by Kahn \cite{KAHN} in 2002 using an ``Entropy'' approach.
Finding a closed-form expression for the $n$th Dedekind number is a very hard problem, also known as Dedekind's problem and its exact values have been found only for $n \leq 8$ \cite{WIEDEMANN} :
    \begin{eqnarray*}
    2, 3, 6, 20, 168, 7581, &&7828354,\hfill\\
    && 2414682040998, 56130437228687557907788 .
    \end{eqnarray*}
This is sequence A000372 in Sloane's Online Encyclopedia of Integer Sequences \cite{SLOANE}.

Monotonic boolean functions on sets of numbers are boolean-valued functions that preserve inclusion $(A \subset B \wedge f(B) \Rightarrow f(A))$.
A monotonic boolean function $f$ is uniquely determined by the largest subset $S$ for which $f(S) = true$.
These largest subsets define another category of functions which we will call antimonotonic.
These correspond to the antichains in the lattice defined by set inclusion \cite{COMTET}.

The main contribution of this paper is an algebra of intervals for the set of antimonotonic functions.
The algebra is based on the natural partial ordering on this set.
Given two comparable antimonotonic functions $\alpha \le \beta$ according to this ordering, the interval $[\alpha,\beta]$ is the set of antimonotonic functions $\gamma$ satisfying $\alpha \le \gamma \le \beta$.
We demonstrate how the space of antimonotonic functions on a finite set, and in fact any of its intervals, can be decomposed as a discrete union of intervals with border elements from lower dimensional subspaces.
This result is based on two decomposition theorems, the second being a consequence of the first.
The first theorem starts from a general antimonootonic function and uses the sets inside this function to generate
the lower dimensional subspaces.
The second theorem uses a partition of the basic set of n elements to genereate the antimonotonic function.
The latter theorem allows decompostion with somewhat different properties.
As applications of the decomposition, we derive a number of recursion formulas for Dedekind numbers.
These formulae are new to the best of our knowledge.
Another application is a class of algorithms for enumeration based on recursive sectioning of the intervals
We prove finiteness of these algorithms and argue that some of these algorithms  are of output-polynomial time complexity.

In section \ref{ordering} we give a definition of the concept of monotonic functions, define partial order on monotonic function and characterise immediate succession.
In section \ref{antimonotonic} we define antimonotonic functions, give the well known isomorphism with the monotonic functions and derive the characterisation of the immediate succession for antimonotonic functions.
In section \ref{sec:operators} we define the well known join and meet operators as well as a convenient projection operator. A fourth operator is introduced which plays an important role in the decomposition theorems that follow.
Apart from the introduction of the latter operator, the first three sections mainly serve to set the notations.
In section \ref{sec:intervals}, we introduce our intervals of antimonotonic functions.
In section \ref{sec:uniformspandecomposition}, the decomposition of an important c lass of intervals - intervals of uniform span - is discussed.
In section \ref{sec:generaldecomposition}, the decomposition of general intervals is studied and an interval based enumeration algorithm is given.
We summarise our results in section \ref{sec:conclusions} and anticipate on further work.
\section{Partial ordering on monotonic functions}
\label{ordering}
Given a positive integer $n$, we denote by $P_n$ the set of the positive integers less than or equal to $n$.
We are interested in boolean functions defined on all subsets of $P_n$, which we denote as $2^{P_n}$.
Any such boolean function $f$ is uniquely defined by the set $f^{-1}(true)$.
We will not distinguish and consider $f$ as a function or as a set whichever is the clearest.
In other words, for a boolean function $f$ and a set $X \subseteq P_n$, we have
\begin{equation}
	X \in f \text{ is equivalent to } f(X).
	\label{eq:functionsassets}
\end{equation}
The following property defines monotonicity.
\begin{definition}{\bf Monotonic boolean functions\\}
A monotonic boolean function on $2^{P_n}$ is a boolean-valued function $f$ such that
\begin{equation}
\label{definition:monotonicfunction}
\forall S \subseteq S' \subseteq P_n:f(S') \Rightarrow f(S).
\end{equation}
We denote by $MT(n)$ the set of monotonic boolean functions on $P_n$.
\end{definition}
\noindent The set $MT(n)$ may be partially ordered by the following natural order relation.
\begin{definition}{\bf Partial ordering of $MT(n)$\\}
Given two monotonic boolean functions $f_1, f_2 \in MT(n)$, we define the partial order relation $\le$ as
\begin{equation}
	f_1 \le f_2 \Leftrightarrow \forall S \subseteq P_n:f_1(S) \Rightarrow f_2(S)
	\label{eq:mototonicordering}
\end{equation}
We define strict inequality in the usual way by
\begin{equation}
	f_1 < f_2 \Leftrightarrow f_1 \not= f_2 \wedge f_1 \le f_2
	\label{eq:strictmonotonicordering}
\end{equation}
We will also use $\ge (>)$:
\begin{equation}
	f \ge (>) g \Leftrightarrow g \le (<) f
	\label{eq:partialgreaterthan}
\end{equation}
\end{definition}
\noindent The immediate successors of a function in $MT(n)$ are defined by
\begin{definition}{\bf Immediate succession in $MT(n)$\\}
Given functions $f, g \in MT(n)$, we say that $g$ is an immediate successor of $f$ iff
\begin{equation}
	f < g \wedge \nexists h \in MT(n):f < h < g
	\label{eq:immediatesuccession}
\end{equation}
We denote the set of immediate successors of $f \in MT(n)$ by $Next(f)$ and if $g \in Next(f)$ we say $f <_{im} g$ or equivalently $g >_{im} f$.
\end{definition}
\noindent The following lemma allows to construct immediate successors in $MT(n)$.
\begin{lemma}
\label{lemma:immediatesuccessors}
Given $f,g \in MT(n)$, we have
\begin{equation}
	f <_{im} g \Leftrightarrow \exists S \in 2^{P_n}:g \backslash f = \{S\} \wedge \forall X \subsetneq S:X \in f
	\label{eq:immediatesuccessionlemma}
\end{equation}
\end{lemma}
\begin{proof}
Clearly $|g \backslash f| > 1 \Rightarrow g \notin Next(f)$ and if $g \backslash f = \{S\}$ with $S' \notin f$ for some $S' \subsetneq S \subset P_n$, then $f < f \cup 2^{S'} < g$ and $g \notin Next(f)$. This proves the ``$\Rightarrow$'' part of the lemma.
Given $f \in MT(n)$, $S \notin f$ such that $\forall X \subsetneq S:X \in f$.
$f \cup \{S\}$ is monotonic and thus in $MT(n)$. We prove that $(f \cup \{S\}) >_{im} f$.
Suppose there is a $h \in MT(n)$ with $f < h < (f \cup \{S\})$.
$h < (f \cup \{S\})$ implies that $\forall X \in h: X \in f \vee X = S$.
Since $f < h$, $X \in f$ implies $X \in h$ so that $(h \not= f \cup \{S\} \Rightarrow S \notin h)$.
But this would imply that $f = h$, a contradiction.
\end{proof}
\noindent Since the smallest monotonic function is $\emptyset$ and the largest is $2^{P_n}$, we have
\begin{corollary}
The length of the longest chains of immediate successors is $2^n+1$.
\end{corollary}

\section{Antimonotonic functions}
\label{antimonotonic}
\begin{definition}{\bf Antimonotonic boolean functions\\}
An antimonotonic boolean function on $2^{P_n}$ is a boolean-valued function $\alpha$ such that
\begin{equation}
\label{definition:antimonotonicfunction}
\forall S \subsetneq S' \subseteq P_n:\alpha(S') \Rightarrow \neg \alpha(S).
\end{equation}
We denote by $AMT(n)$ the set of antimonotonic boolean functions on $2^{P_n}$.
\end{definition}

\begin{theorem}
$\forall n > 0:|AMT(n)| = |MT(n)|$
\end{theorem}
\begin{proof}
Consider the boolean function $mt_n:AMT(n) \rightarrow MT(n)$ defined by
\[
\forall \alpha \in AMT(n):mt_n(\alpha) = \{S \subseteq P_n|\exists S' \in \alpha:S \subseteq S'\}
\]
The function $mt_n(\alpha)$ is monotonic by definition.
Given $f \in MT(n)$, $mt_n^{-1}(f)$ is given by the set of maximal elements of $f$.
\[
mt_n^{-1}(f) = \{S \in f|\neg \exists S' \in f:S \subset S'\}
\]
which is in $AMT(n)$ and is the inverse of $mt_n$ as can be verified.
\end{proof}
\noindent We will denote by $mt_n$ the mapping from $AMT(n)$ to $MT(n)$ and by $amt_n$ its inverse.
\\
\\
The ordering on $MT(n)$ induces an ordering on $AMT(n)$:
\begin{definition}{\bf Partial ordering of $AMT(n)$\\}
Given two antimonotonic boolean functions $\alpha, \beta \in AMT(n)$, $\alpha \le \beta$ iff $mt_n(\alpha) \le mt_n(\beta)$
and $\alpha < \beta$ iff $mt_n(\alpha) < mt_n(\beta)$.
\end{definition}
\begin{lemma}{\bf Partial ordering of $AMT(n)$\\}
The induced ordering on $AMT(n)$ is given by
\begin{equation}
	\forall \alpha,\beta \in AMT(n):\alpha \le \beta \\
	\Leftrightarrow \forall S \in \alpha:\exists S' \supseteq S:S' \in \beta
	\label{eq:antimototonicordering}
\end{equation}
The strict ordering is given by
\begin{equation}
	\forall \alpha,\beta \in AMT(n):\alpha < \beta \\
	\Leftrightarrow \alpha \le \beta \wedge \exists S \in \beta:S \notin \alpha
	\label{eq:strictantimototonicordering}
\end{equation}

\end{lemma}
\noindent The definition of immediate successors in $AMT(n)$ follows the corresponding definition for $MT(n)$.
As in $MT(n)$, $next(\alpha)$ is the set of immediate successors of $\alpha \in AMT(n)$.
An immediate successor is constructed according to the following lemma.
\begin{lemma}
\label{lemma:antimonotoneimmediatesuccessors}
Given $\alpha,\beta \in AMT(n)$, we have
\begin{equation}
	\beta >_{im} \alpha \Leftrightarrow \alpha < \beta \ \wedge\ \exists S \in 2^{P_n}:\beta \backslash \alpha = \{S\} \wedge \forall X \subsetneq S:\exists X' \supseteq X: X' \in \alpha
	\label{eq:antimonotoneimmediatesuccessionlemma}
\end{equation}
\end{lemma}
\begin{proof}
Let $\beta$ satisfy the condition of the lemma with $S$ the only set contained in $\beta$ and not in $\alpha$.
We prove that $mt_n(\beta) >_{im} mt_n(\alpha)$.
It is easily seen that $\alpha$ contains all elements of $\beta$ except $S$ and real subsets of $S$ with cardinality $|S|-1$.
So it follows that $mt_n(\beta) \backslash mt_n(\alpha) = \{S\}$.
Let $X \in mt_n(\beta)$ such that $X \not\subseteq S$.
$X$ clearly is in $mt_n(\alpha)$.
Let $X \not= S$ be a subset of S.
Again it follows that $X \in mt_n(\alpha)$.
So $mt_n(\beta)$ satisfies the conditions of lemma \ref{lemma:immediatesuccessors} and is an immediate successor
of $mt_n(\alpha)$.
It follows that $\beta >_{im} \alpha$.
This proves the leftward side of the equivalence. The proof of the other side is similar.
\end{proof}
\begin{definition}{\bf Extension of the notation\\}
Let $n_1,n_2 \in \mathbb{N},\ n_1\leq n_2$. We will use the notation $AMT(n_1,n_2)$ for the space of antimonotonic functions on subsets of $\{n_1,...,n_2\}$.
Clearly, we have $AMT(n) = AMT(1,n)$.
More generally, for any finite set $M$ of natural numbers, $AMT(M)$ will denote the the space of antimonotonic functions on subsets of $M$.
\end{definition}
%

\section{Projection and other operators}
\label{sec:operators}
\subsection{Projection}
Given two finite sets $N' \subseteq N$ the projection from the space $AMT(N)$ to $AMT(N')$ is defined by
\begin{equation}
	\pi_{N'}:AMT(N) \rightarrow AMT(N'):\alpha \rightarrow sup(\{A \cap N'|A \in \alpha\})
	\label{def:projection}
\end{equation}
\begin{prop}[Order conservation by projection]
For sets of integers $N' \subseteq N$, $\alpha, \beta \in AMT(N)$ we have
\begin{equation}
	\alpha \le \beta \Rightarrow \pi_{N'}(\alpha) \le \pi_{N'}(\beta)
	\label{prop:orderconservationprojection}
\end{equation}
\end{prop}

\noindent Three further operators are useful to explore the space of antimonotonic functions further. 
The first two immediately follow from the meet and join operators in the lattice of anti-chains of which the anti-monotonic functions are a representation.
The definition of the third operator given here is tied to the specifics of anti-monotonic functions, but as we demonstrate in appendix \ref{app:generallattices}, this operator and its applications in the intervals we subsequently derive can be readily generalised to complete distributive lattices. 
\subsection{Meet}
We define the idempotent meet operator $(\wedge)$ as follows
\begin{equation}
	\forall \alpha,\beta \in AMT(N):\alpha \wedge \beta = sup(\{A \cap B| A \in \alpha, B \in \beta\}))
	\label{def:externaldot}
\end{equation}
Where de operation $sup$ is defined as follows
\begin{equation}
\forall S \subset 2^N:sup(S) = \{A \in S|\nexists A' \in S: A \subsetneq A'\}
\label{def:sup}
\end{equation}
Obviously we have that $\forall \alpha,\beta \in AMT(N):\alpha \wedge \beta \in AMT(N)$. 
The following property makes the connection with the more general operator for anti-chains.
\begin{prop}[Largest common lower bound]
For antimonotonic functions $\alpha,\beta,\kappa$ we have
\begin{equation}
	\kappa \le \alpha\ and\ \kappa \le \beta \Leftrightarrow \kappa \le \alpha \wedge \beta
	\label{prop:commonlowerbound}
\end{equation}
\begin{proof}
($\Rightarrow$) For any $K \in \kappa$ we have sets $A \in \alpha, B \in \beta$, such that $K \subseteq A \cap B$.\\
($\Leftarrow$) For any $K \in \kappa$, pick $A \cap B \in \alpha \wedge \beta$ such that $K \subseteq A \cap B$. Clearly $K \subseteq A$ and $K \subseteq B$.
\end{proof}
\end{prop}
$\alpha \wedge \beta$ is the largest anti-monotonic function that is smaller than both $\alpha$ and $\beta$.
\subsection{Join}
We define the idempotent join operator $\vee$ as follows
\begin{equation}
\forall \alpha,\beta \in AMT(1,n):\alpha \vee \beta = sup(\alpha \cup \beta)
\label{def:externalsum}
\end{equation}
$\alpha \vee \beta$ is antimonotonic.
One finds for $\alpha \neq \emptyset$:
\begin{eqnarray}
\alpha \vee \emptyset & = & \alpha \nonumber\\
\alpha \vee \{\emptyset\} & = & \alpha
\label{prop:externalsum}
\end{eqnarray}
\begin{prop}[Least common upper bound]
For antimonotonic functions $\alpha,\beta,\kappa$ we have
\begin{equation}
	\kappa \ge \alpha\ and\ \kappa \ge \beta \Leftrightarrow \kappa \ge \alpha \vee \beta
	\label{prop:commonupperbound}
\end{equation}
\end{prop}
$\alpha \vee \beta$ is the smallest anti-monotonic function dominating both $\alpha$ and $\beta$.
\subsection{External product}
We define the span of an anti-monotonic function $\alpha$ as
\begin{definition}
$sp(\alpha) = \cup_{A \in \alpha}{(A)}$ 
\end{definition}
The external product with respect to subsets of $N$ is defined as follows
\label{sec:externalproduct}
%
\begin{definition}
Let $N$ be a set of integers, 
$\alpha,\beta \in AMT(N)$ 
\noindent The external product of $\alpha$ and $\beta$ is given by
\begin{equation}
\alpha \times \beta = max\{\kappa \in AMT(N): sp(\kappa) = sp(\alpha) \cup sp(\beta), \pi_{sp(\alpha)}(\kappa) \le \alpha, \pi_{sp(\beta)}(\kappa) \le \beta\}
\label{def:externalproduct}
\end{equation}
\end{definition} 
\noindent The maximum in definition (\ref{def:externalproduct}) is unique as is shown by the construction
\begin{equation}
\alpha \times \beta = sup\{(A \backslash sp(\beta)) \cup (B \backslash sp(\alpha)) \cup (A \cap B)|A \in \alpha, B \in \beta\}
\label{def:externalproductb}
\end{equation}
The operation $\times$ has the associative and commutative properties. It is idempotent. Its neutral element is $\{\emptyset\}$ and $\emptyset$ is the annihilating element.
The following important property is an immediate consequence of the definition.
\begin{prop}
For $\alpha,\beta,\gamma \in AMT(N)$ we have
\begin{equation}
\label{prop:externalproductseparation}
\gamma \le \alpha \times \beta \Rightarrow \pi_{sp(\alpha)}(\gamma) \le \alpha
\end{equation}
\end{prop}
%
\section{Intervals of antimonotonic functions}
\label{sec:intervals}
For $\alpha \le \beta \in AMT(N)$ we define the intervals
\begin{equation}
{[\alpha,\beta]} = \{\kappa \in AMT(N)| \alpha \le \kappa \le \beta\}.
\end{equation}
\begin{prop}[Intersection of intervals]
For $\alpha,\beta,\alpha',\beta' \in AMT(n)$ the intersection of the intervals $[\alpha,\beta]$ and $[\alpha',\beta']$ is given by
\begin{equation}
	[\alpha,\beta] \cap [\alpha',\beta'] = [\alpha \vee \alpha',\beta \wedge \beta']
	\label{prop:intersectionofintervals}
\end{equation}
\end{prop}
\noindent For $\alpha,\beta \in AMT(n)$ we consider intervals of the form
\begin{equation}
[\alpha \vee \beta,\alpha \times \beta]
\end{equation}
Because $\emptyset$ is the annihilating element and $\{\emptyset\}$ the unit element for the external product
we have
\begin{equation}
[\alpha \vee \emptyset,\alpha \times \emptyset] = \emptyset, 
[\alpha \vee \{\emptyset\},\alpha \times \{\emptyset\}] = \{\alpha\}
\label{prop:intervalemptyset}
\end{equation}
%
Since $sp(\alpha \vee \beta) = sp(\alpha \times \beta)$, we have 
\begin{prop}[Uniform span]
For $\alpha, \beta \in AMT(n)$, each element $\kappa \in [\alpha \vee \beta,\alpha \times \beta]$ has the same span:
\begin{equation}
	\forall \kappa \in [\alpha \vee \beta,\alpha \times \beta]: 
	sp(\kappa) = sp(\alpha \vee \beta)
	\label{prop:uniformspan}
\end{equation}
\end{prop}
%
\noindent Since $\alpha = \pi_{sp(\alpha)}(\alpha) \leq \pi_{sp(\alpha)}(\alpha \vee \beta) $ and due to the inequality (\ref{prop:externalproductseparation}) for all elements in the interval   $ [\alpha \vee \beta,\alpha \times \beta]$ the following invariant holds.
\begin{prop}
\label{prop:intervalinvariant}
For each $\kappa \in [\alpha \vee \beta,\alpha \times \beta]$ we have $\pi_{sp(\alpha)}(\kappa) = \alpha, \pi_{sp(\beta)}(\kappa) = \beta$.
	\label{prop:intervalinequality}
\end{prop}
\noindent These definitions and properties are readily generalised for an arbitrary number of antimonotonic functions $\alpha_1,\dots,\alpha_k$ considering intervals of the type $[\alpha_1 \vee \dots \vee \alpha_k,\alpha_1\times \dots\times\alpha_k]$. To simplify the notation we will use
\begin{eqnarray}
\alpha_1 \vee \dots \vee \alpha_k &=& \vee\{\alpha_1,\dots,\alpha_k \} \nonumber \\
\alpha_1 \times \dots \vee \alpha_k &=& \times\{\alpha_1,\dots,\alpha_k \}
\end{eqnarray}
For a finite set $N$ of integers, we define
\begin{eqnarray}
{\alpha}_N &\equiv& \vee\{\{x\}|x \in N\} = \{\{x\}|x \in N\} \label{def:minalpha} \\
{\omega}_N &\equiv& \times\{\{x\}|x \in N\} = \{N\}\label{def:maxalpha}\\
{\Upsilon}_N & \equiv & [\alpha_N,\omega_N]
\end{eqnarray}
We find $sp({\alpha}_N) = sp({\omega}_N) = N$. The interval $\Upsilon_N = [{\alpha}_N,{\omega}_N]$ is the set of antimonotonic functions with span $N$. We investigate the structure of this interval in the next paragraph.
%

\section{Decomposition of the interval $\Upsilon_N$}
\label{sec:uniformspandecomposition}
The special intervals $[\alpha \vee \beta,\alpha \times \beta]$ for $\alpha,\beta \in AMT(N)$ from section \ref{sec:intervals} allow decomposing the interval $\Upsilon_N$, effectively reducing its dimensionality. That is why we refer to the following theorem as the ``General coordinate system'' theorem.
\begin{theorem}[General coordinate system]
\label{the:generalcoordinatesystem}
Given a finite set of positive integers $N$ and an antimonotonic function $\sigma$ with $sp(\sigma) = N$,
the set of nonempty intervals of the type
\[
{[\vee\{\kappa_S|S \in \sigma\},\times\{\kappa_S|S \in \sigma\}]}
\]
where $\{\kappa_S|S \in \sigma\}$ is any family of antimonotonic functions satisfying $\forall S \in \sigma:\kappa_S \in \Upsilon_S$, form a partition of $\Upsilon_N$.
\begin{proof}
In order to prove that the intervals from this set are disjoint, consider two families $\{\kappa_S|S \in \sigma\}$  and $\{\kappa'_S|S \in \sigma\}$. Let $S_1 \in \sigma$ be any set in the antimonotonic function $\sigma$ and let $\gamma \in {[\vee\{\kappa_S|S \in \sigma\},\times\{\kappa_S|S \in \sigma\}]} \cap {[\vee\{\kappa'_S|S \in \sigma\},\times\{\kappa'_S|S \in \sigma\}]}$ be any element in the intersection of the two intervals. We find
\begin{eqnarray}
\kappa_{S_1} & \le & \pi_{S_1}(\gamma) \le \kappa'_{S_1} \nonumber \\
\kappa'_{S_1} & \le & \pi_{S_1}(\gamma) \le \kappa_{S_1}
\end{eqnarray}
and conclude that $\kappa_{S_1} = \kappa'_{S_1}$ if the intervals have at least one element in common.
Two intervals are thus either equal or disjoint.
\\
\\
The union of all intervals is $\Upsilon_N$ since each $\gamma \in \Upsilon_N$ satisfies
\[
\gamma \in [\vee\{\pi_{S}(\gamma)|S \in \sigma\},\times \{\pi_{S}(\gamma)|S \in \sigma\}]
\]
Since for any $\gamma \in \Upsilon_N$, $sp(\gamma) = N$, we find that for any $S \in \sigma$, $sp(\pi_S(\gamma) )= S$ and hence $\pi_S(\gamma) \in \Upsilon_S$. This completes the proof.
\end{proof}
\end{theorem}

\begin{remark}
\label{rem:restrictioncoordinatesystem}
The restriction to antimonotonic functions $\sigma$ in Theorem \ref{the:generalcoordinatesystem} is not necessary for the theorem to hold.
If however, in the notation of Theorem \ref{the:generalcoordinatesystem}, $S_1 \subsetneq S_2$, we have 
for $\kappa_{S_1} \in \Upsilon_{S_1}, \kappa_{S_2} \in \Upsilon_{S_2}$: 
\begin{eqnarray}
\kappa_{S_1} \vee \kappa_{S_2} & \in  & \Upsilon_{S_2} \nonumber \\
\kappa_{S_1} \times \kappa_{S_2} & = & \kappa_{S_1} \wedge \kappa_{S_2} \le \kappa_{S_1} \vee \kappa_{S_2}  \nonumber
\end{eqnarray}
with equality only if $\kappa_{S_1} = \pi_{S_1}(\kappa_{S_2})$. 
Consequently, any interval in the decomposition can be nonempty only if the latter condition is satisfied, leaving no freedom for the choice of $\kappa_{S_1}$.
\end{remark}

\begin{remark}
\label{rem:notnonemptycoordinatesystem}
The restriction does not remove all empty intervals from the decomposition.
If, still in the notation of Theorem \ref{the:generalcoordinatesystem}, two sets $S_1, S_2 \in \sigma$ are not disjoint and
\begin{equation}
\label{eq:interactioncoordinatesystem}
\pi_{S_1 \cap S_2}(\kappa_{S_1}) \neq \pi_{S_1 \cap S_2}(\kappa_{S_2})
\end{equation}
we find that
\[
\pi_{S_1 \cap S_2}(\kappa_{S_1} \vee \kappa_{S_2} ) \not\le 
\pi_{S_1 \cap S_2}(\kappa_{S_1} \times \kappa_{S_2} )
\]
and consequently
\[
\kappa_{S_1} \vee \kappa_{S_2} \not\le 
\kappa_{S_1} \times \kappa_{S_2}
\]
making an empty interval.
\end{remark}

\noindent Although the general form of Theorem \ref{the:generalcoordinatesystem} allows for studying symmetries in the decomposition,
the next theorem removes the empty intervals from Remark \ref{rem:notnonemptycoordinatesystem} by requiring $\sigma$ to contain only disjoint sets, i.e. to be a partition of $N$.
This restriction avoids the interaction expressed by condition (\ref{eq:interactioncoordinatesystem})
between the lower dimensional subspaces, hence its name ``orthogonal coordinate system theorem''.

\begin{theorem}[Orthogonal coordinate System]
\label{the:orthogonalcoordinatesystem}
For a finite set of positive integers $N$ and a partition $\sigma$ of this set, the set of intervals of the type
\[
{[\vee\{\kappa_S|S \in \sigma\},\times\{\kappa_S|S \in \sigma\}]}
\]
where $\{\kappa_S|S \in \sigma\}$ is any family of antimonotonic functions satisfying $\forall S \in \sigma:\kappa_S \in \Upsilon_S$,
form a partition of $\Upsilon_N$.
\begin{proof}
This theorem is a special case of Theorem \ref{the:generalcoordinatesystem}.
\end{proof}
\end{theorem}

\ \\
\noindent Given two finite sets $N1 \neq N2$, the corresponding intervals of uniform span $\Upsilon_{N_1}$ and
$\Upsilon_{N_2}$ are disjoint.
The family $(\Upsilon_S|S \subseteq N)$ forms a partition of  $AMT(N) \backslash \{\emptyset\}$ and we have the disjoint union
\begin{equation}
AMT(N) = \{\emptyset\} \cup (\bigcup_{S \subseteq N} \Upsilon_S)
\label{eq:fulldecomposition}
\end{equation}

\begin{corollary}
The following expansion is an immediate consequence
\begin{equation}
\label{eq:uniformspanexpansion}
|AMT(n)| = 1
+   \binom{n}{0}|\Upsilon_{\emptyset}|         
+   \binom{n}{1} |\Upsilon_{\{1\}}|                  
+   \binom{n}{2} |\Upsilon_{\{1,2\}}|                                         
+  \dots                                                                                                    
+   |\Upsilon_{\{1,\dots,n\}}|                                         
\end{equation}
\end{corollary}
\noindent Equation (\ref{eq:fulldecomposition}) leads to the following decomposition of $AMT(N)$.
\begin{corollary}
\begin{equation}
AMT(N) = \{\emptyset\} \cup ( \bigcup_{(\kappa_S \in AMT(S)\backslash \{\emptyset\}|S \in \sigma)}[\vee \{\kappa_S|S \in \sigma\},\times \{\kappa_S|S \in \sigma \}])
\end{equation}
where the union is taken over all families of antimonotonic functions from $AMT(S)$, one for each $S \in \sigma$.
\end{corollary}
\begin{corollary}
A consequence of the decomposition in equation (\ref{eq:fulldecomposition}) is
\begin{equation}
|AMT(N)| = 1 + \sum_{S \subseteq N} |\Upsilon_S|
\end{equation}
\label{cor:fulldecomposition}
\end{corollary}

\noindent The two theorems allow to derive a number of recursion relations for Dedekind-like numbers.
The first one is an immediate consequence.
\begin{corollary}\label{cor:uniformspanintervalcounting}
For positive integers $n_1$ and $n$ such that $1 \le n_1 < n$ we have
\begin{equation}
|\Upsilon_{\{1,\dots, n\}}| = \sum_{\alpha \in \Upsilon_{\{1,\dots ,n_1\}},
\beta \in \Upsilon_{\{n_1+1,\dots, n\}}}{|[\alpha \vee \beta,\alpha \times \beta]|}
\end{equation}
\end{corollary}
\noindent Slight rearrangement in the sums allows to derive the following recursion relation.
\begin{corollary}\label{cor:intervalrecursioncounting}
For positive integers $n_1$ and $n$ such that $1 \le n_1 < n$ we have
\begin{equation}
|AMT(1,n)| = 1+\sum_{\alpha \in AMT(1,n_1)\backslash\{\emptyset\},
\beta \in AMT(n_1+1,n)\backslash\{\emptyset\}}{|[\alpha \vee \beta,\alpha \times \beta]|}
\end{equation}
\end{corollary}
\noindent An interesting special case of Corollary \ref{cor:intervalrecursioncounting} is $n_1 = n-1$:
\begin{corollary} For any positive integer $n > 1$
\begin{equation}
|AMT(1,n)| = \sum_{\alpha \in AMT(1,n-1)}{|[\emptyset,\alpha]|}	
\label{cor:oneelementrecursioncounting}
\end{equation}
\begin{proof}
Let $n_1 = n - 1$. We have:
\begin{equation}
AMT(n_1+1,n) \backslash \{\emptyset\} = AMT(n,n) \backslash \{\emptyset\} = \{\{\emptyset\},\{n\}\}
\end{equation}
and Corollary \ref{cor:intervalrecursioncounting} implies
\begin{equation}
|AMT(1,n)| = 1+\sum_{\alpha \in AMT(1,n-1)\backslash\{\emptyset\},
\beta \in \{\{\emptyset\},\{\{n\}\}\}}{|[\alpha \vee \beta,\alpha \times \beta]|}
\end{equation}
Since $\{\emptyset\}$ is neutral for $\vee$ and $\times$, this is equivalent to
\begin{equation}
|AMT(1,n)| = |AMT(1,n-1)| + \sum_{\alpha \in AMT(1,n-1)\backslash\{\emptyset\}}{|[\alpha \vee \{\{n\}\},\alpha \times \{\{n\}\}]|}
\end{equation}
Each element $\kappa$ of $[\alpha \vee \{\{n\}\},\alpha \times \{\{n\}\}]$ can be written as $\alpha \vee (\kappa' \times \{\{n\}\})$ for some $\kappa' \in [\{\emptyset\},\alpha]$.\footnote{
$sp(\kappa) = sp(\alpha) \cup \{n\}$, $\kappa' = \pi_{sp(\alpha)}(\kappa \backslash \alpha)$.}
Hence
\begin{equation}
|[\alpha \vee \{\{n\}\},\alpha \times \{\{n\}\}]| = |[\{\emptyset\},\alpha]|
\label{cor:intervalsizeoneelement}
\end{equation}
and we find
\begin{eqnarray}
|AMT(1,n)| &=& |AMT(1,n-1)| + \sum_{\alpha \in AMT(1,n-1)\backslash\{\emptyset\}}{|[\{\emptyset\},\alpha]|} \nonumber \\
&=& 1 + \sum_{\alpha \in AMT(1,n-1)\backslash\{\emptyset\}}{|[\emptyset,\alpha]|} \nonumber \\
&=& \sum_{\alpha \in AMT(1,n-1)}{|[\emptyset,\alpha]|} 	
\end{eqnarray}
\end{proof}
\end{corollary}
\section{Structure of general intervals}
\label{sec:generaldecomposition}
For a finite set $N$ and $\alpha,\alpha',\beta,\beta' \in AMT(N)$, the intersection of intervals $[\alpha('),\beta(')]$ is given by
\begin{equation}
[\alpha,\beta] \cap [\alpha',\beta'] = [\alpha \vee \alpha',\beta \wedge \beta']
\end{equation}
The following theorem is an immediate consequence of Corollary \ref{cor:fulldecomposition}.
\begin{theorem}[Decomposition of intervals]
Given an interval $[\alpha,\omega]$ of antimonotone functions with $\alpha \neq \emptyset$ and an antimonotonic funtion $\sigma \le \omega$ with $sp(\sigma) = sp(\omega)$ then
the interval $[\alpha,\omega]$ is the disjoint union of intervals 
\[
{[\vee\{\kappa_S|S \in \sigma\} \vee \alpha,\times\{\kappa_S|S \in \sigma\} \wedge \omega]}
\]
where each $\kappa_S \in [\pi_{S}(\alpha),\pi_{S}(\omega)]$.
\label{the:intervaldecomposition}
\end{theorem}

\subsection{Ranks and distances}
One of the main results in this article is a recursive procedure to split an interval in ever smaller subintervals to arrive at an enumeration of its elements.
We need one more device before we can present the procedure. 
The device is a measure for the distance between two antimonotonic functions.
We will use this measure for a heuristic estimate of the size of an interval,
and to ensure that after the split, the fragments are smaller than the whole.
We first introduce the rank of an antimonotonic function. It is defined as the number of different subsets of sets in the function:
\begin{definition}
\begin{equation}
\forall \alpha \in AMT(N):
rank(\alpha) = \sum_{A \in \alpha}{2^{|A|}} 
	-  \sum_{A\neq B \in \alpha}{2^{|A \cap B|}} 
	+  \sum_{A\neq B\neq C \in \alpha}{2^{|A \cap B \cap C|}} -...
\label{def:rankfunction}
\end{equation}
\end{definition}

\noindent The distance between two antimonotonic functions is defined as 
\begin{definition}
\begin{equation}
\forall \alpha,\beta \in AMT(N): 
d(\alpha,\beta) = rank(\alpha) + rank(\beta) - 2\times rank(\alpha\wedge\beta)
\label{def:distancefunction}
\end{equation}
\end{definition}
%
\subsection{Recusively partitioning intervals}
The decomposition in Theorem \ref{the:intervaldecomposition} can be repeated recusively to enumerate antimonotonic functions over a complete space or within an interval.
It is not hard to see that for $\emptyset < \alpha < \omega$, there is always a set $N \subseteq sp(\omega)$
allowing to produce smaller intervals.
\begin{prop}
For antimonotonic functions $\emptyset < \alpha < \omega$, 
there is always a nonempty set $N \subseteq sp(\omega)$ such that
$\pi_{N}(\alpha) < \pi_{N}(\omega)$. If $|\omega| > 1$ or $\alpha$ is not an immediate predecessor of $\omega$, 
there is such a set satisfying $N \subsetneq sp(\omega)$ 
\label{prop:descentproperty}
\begin{proof}
$N=sp(\omega)$ satisfies the first condition. Another example is found as follows. Let $A \in \omega$ be a non-empty set such that $A \notin \alpha$. 
It can be seen that $\pi_A(\alpha) < \pi_A(\omega)$.
(Should $P_A(\alpha) = P_A(\omega)$, then there would be a set $B \in \alpha$ 
with $A \subsetneq B$ contradicting $\alpha < \omega$.)
If $|\omega| > 1$ then $|sp(\omega)| > |A|$ and $sp(\omega) \backslash A \not= \emptyset$.
If $|\omega| = 1$ and $\alpha$ is not an immediate predecessor of $\omega$
then there must be a set $X \notin \alpha$ of the form $sp(\omega) \backslash \{a\}$ with $a$ an element of the only set in $\omega$.
We find that $P_X(\alpha) < P_X(\omega)$ and $sp(\omega) \backslash X$ is non empty. 
\end{proof}
\end{prop}
\noindent Given that $\alpha <_{im} \omega \Leftrightarrow |[\alpha,\omega]| = 2$, 
and building on Property \ref{prop:descentproperty}, 
Procedure \ref{proc:intervallisting} lists the elements of an interval $[\alpha,\beta]$.
\begin{algorithm}
\caption{List all elements of an interval of antimonotonic functions and return their number}
\begin{algorithmic}
\label{proc:intervallisting}
\REQUIRE $\alpha \not= \emptyset, \beta \not= \emptyset$
\ENSURE $|[\alpha,\beta]|$
\ENSURE print all elements of $[\alpha,\beta]$
\STATE {\bf function} listElements($\alpha,\beta \in AMT(n)$) {\bf returns} number 
\IF{$\alpha \not\le \beta $}
	\RETURN 0
\ENDIF
\IF {$\alpha = \beta$ } 
	\PRINT $\alpha$
	\RETURN 1
\ENDIF
\IF {$|b| = 1\ and\ \alpha <_{im} b$}
 	\PRINT $\alpha, \beta$
	\RETURN 2
\ENDIF
\STATE select disjoint subsets X,Y of $sp(\beta)$ such that
\STATE $d(P_X(\alpha),P_X( \beta))\ge 1\ and\ d(P_Y( \alpha),P_Y( \beta)) \ge 1$ 
\STATE {\bf set} COUNT = 0
\FORALL{$\kappa \in [P_X(\alpha),P_X(\beta)]$}
	\FORALL{$\lambda \in [P_Y(\alpha),P_Y(\beta)]$}
			\STATE {\bf set} COUNT = COUNT + listElements($(\kappa \vee \lambda) \vee \alpha, (\kappa \times \lambda) \wedge \beta $)
	\ENDFOR
\ENDFOR
\RETURN COUNT
\STATE {\bf end function}
\end{algorithmic}
\end{algorithm}
Note that the recursion is not only over the fragment $[(\kappa \vee \lambda) \vee \alpha,(\kappa \times \lambda)\wedge \beta]$ but also in the iterations over $[P_X(\alpha),P_X(\beta)]$ and $[P_Y(\alpha),P_Y(\beta)]$. 
Since each couple $\kappa, \lambda$ selected from these intervals are used in the construction of the elements of the new interval, and since finding a suitable split is linear in the size of $span(\omega)$, this procedure is of complexity $|span(\omega)|*outputsize$. Given the double exponential size of the output with respect to $span(\omega)$, we can say that the procedure is essentially linear in the size of the output.

\noindent The conditions $d(P_X(\alpha),P_X( \beta)) \ge 1$ and $d(P_Y( \alpha),P_Y( \beta)) \ge 1$ can be replaced by a condition that splits the interval in parts that are as equal as possible. This reduces the depth of the recursion and has a beneficial effect on the construction of the intervals (bigger $\kappa, \lambda$). It does bring a cost however in the computation of the split of $span(\omega)$. It is not immediately clear what the impact on the complexity is.
%
%
%
%
\section{Conclusions}
\label{sec:conclusions}
We developed an interval algebra for the lattice of anti-monotonic functions on subsets of a finite set.  Formulae for the intersection of intervals are given. An essential operator in the algebra is the external product operator of section \ref{sec:operators}.Together with the join operator, it allows for the decomposition as the union of disjoint intervals of the lattice of anti-monotonic functions on the subsets of a given finite set. This decomposition may be based on the elements of the sub-lattices defined on the sets accepted by any anti-monotonic function that covers the original set. A particularly interesting case of this decomposition relates to an anti-monotonic function in which all sets are disjoint. The external product operator then takes a particularly simple form. Given the rules for the intersection of intervals, the decomposition of the whole lattice naturally leads to a decomposition any interval. The partitions of the lattice as well as of its intervals resulting from this decomposition lead to a number of recursion formulae which can be used to count the number of elements in the lattice or interval (Dedekind problem). As an application of the algebra, we present an algorithm that enumerates all elements of the lattice in output-polynomial time.  An implementation of this algorithm is available.
 In appendix \ref{app:generallattices} we define the operator and prove the partitioning theorem for complete distributive lattices. Appendix \ref{app:example} describes the example of Young's lattice.
\section{Appendix}
\label{app:generallattices}
Let L be a complete distributive lattice with unit element 1.
Let  $\alpha, \beta, \gamma \in L$ and define
\begin{itemize}
\item $Next(\alpha) = \{\kappa|\alpha <_{im} \kappa\}$
\item $base(\alpha) = \alpha \wedge(\vee(Next(1)))$
\item $top(\alpha) = \vee \{\kappa|base(\kappa) = base(\alpha)\}$
\item $\alpha \times \beta = \vee \{\kappa| top(\alpha) \wedge \kappa \le \alpha\ and\  \ top(\beta)\wedge\kappa \le \beta\}$
\end{itemize}
It is straightforward to prove the following
\begin{lemma}
\label{lem:externalproductandorder}
\[
\gamma \le \alpha \times \beta \Rightarrow \gamma \wedge top(\alpha) \le \alpha
\]
\begin{proof}
\begin{eqnarray*}
\gamma & \le & \vee \{\kappa| top(\alpha) \wedge \kappa \le \alpha \ and\  top(\beta)\wedge\kappa \le \beta\}\\
&\Leftrightarrow&\\
\gamma & = & \gamma \wedge (\vee \{\kappa| top(\alpha) \wedge \kappa \le \alpha\ and\  top(\beta)\wedge\kappa \le \beta\})\\
& = & \vee (\{\gamma \wedge \kappa| top(\alpha) \wedge \kappa \le \alpha\ and\  top(\beta)\wedge\kappa \le \beta\})\\
&\Rightarrow&\\
\gamma \wedge top(\alpha)& = & \vee (\{\gamma \wedge top(\alpha) \wedge \kappa| top(\alpha) \wedge \kappa \le \alpha\ and\  top(\beta)\wedge\kappa \le \beta\})\\
& \le & \alpha
\end{eqnarray*}
$\Box$
\end{proof}
\end{lemma}
\noindent Lemma \ref{lem:externalproductandorder} allows proving the following
\begin{theorem}
Let $\alpha, \beta \in L, \kappa_\alpha,\kappa'_\alpha \in [base(\alpha),top(\alpha)], \kappa_\beta,\kappa'_\beta \in [base(\beta),top(\beta)]$. We have
\begin{equation*}
[\kappa_\alpha \vee \kappa_\beta,\kappa_\alpha \times \kappa_\beta] \cap [\kappa'_\alpha \vee \kappa'_\beta,\kappa'_\alpha \times \kappa'_\beta]  \neq \emptyset
\Leftrightarrow
\kappa_\alpha = \kappa'_\alpha,\kappa_\beta = \kappa'_\beta.
\end{equation*}
\begin{proof}
Let $\gamma \in [\kappa_\alpha \vee \kappa_\beta,\kappa_\alpha \times \kappa_\beta] \cap [\kappa'_\alpha \vee \kappa'_\beta,\kappa'_\alpha \times \kappa'_\beta] $. We find e.g.
\begin{eqnarray*}
\kappa_\alpha \le & \gamma \wedge top(\alpha) & \le \kappa'_\alpha\\
&and&\\
\kappa'_\alpha \le & \gamma \wedge top(\alpha) & \le \kappa_\alpha
\end{eqnarray*}
$\Box$
\end{proof}
\end{theorem}
\noindent Since for $\gamma \in [base(\alpha)\vee base(\beta), top(\alpha)\times top(\beta)], 
\gamma \in [(\gamma \wedge top(\alpha)) \vee (\gamma \wedge top(\beta)), (\gamma \wedge top(\alpha)) \times (\gamma \wedge top(\beta))]$, we have constructed a partition of the interval 
 $[base(\alpha)\vee base(\beta), top(\alpha)\times top(\beta)]$.

\section{Appendix: an example}
\label{app:example}
Young's lattice is the lattice of Young diagrams. It is infinite. To apply our theory, let us initially use bounds for the horizontal $(n_h)$ and the vertical $(n_v)$ dimensions of the Young diagrams.
We denote by $vs_i$ vertical strip of size $i$ and by $hs_i$ the horizontal strip of size $i$.
The unit Young diagram is then $u = vs_1 = hs_1$.
The immediate successors of $u$ are $vs_2$ and $hs_2$.
We find that $vs_i \vee hs_j$ is an L-shaped Young diagram with vertical dimension $i$ and horizontal dimension $j$.
The external product $vs_i \times hs_j$ is a rectangular Young diagram with the same dimensions.
According to the decomposition theorem in appendix \ref{app:generallattices}, the set of all nonempty Young diagrams with the given bounds is given by the disjoint union
\begin{equation*}
\bigcup_{0<i<n_v,0<j<n_h}[vs_i \vee hs_j,vs_i \times hs_j].
\end{equation*}
Since this partition hods for all values of $n_v$ and $n_h$, the infinite lattice of nonempty Young diagrams is given by the disjoint union
\begin{equation*}
\bigcup_{0<i,0<j}[vs_i \vee hs_j,vs_i \times hs_j].
\end{equation*}

\end{document}